\documentclass[a4paper]{article}
\usepackage{amsthm,amsfonts,amsmath,amssymb,units}
\usepackage[abbrev,nobysame]{amsrefs}
\usepackage[cp1251]{inputenc}
\usepackage[english]{babel}
\usepackage[final]{graphicx}
\usepackage{setspace}
\usepackage[12pt]{extsizes}
\begin{document}
%\large
\newtheorem{teorema}{Theorem}
\newtheorem{lemma}{Lemma}
\newtheorem{utv}{Proposition}
\newtheorem{svoistvo}{Property}
\newtheorem{sled}{Corollary}
\newtheorem{con}{Conjecture}
\newtheorem{zam}{Remark}
\newtheorem{quest}{Question}

\author{V. N. Potapov\thanks{\texttt{quasigroup349@gmail.com}}, A. A. Taranenko\thanks{Sobolev Institute of Mathematics, Novosibirsk, Russia; \texttt{taa@math.nsc.ru}}}
\title{Asymptotic bounds on the numbers of vertices of polytopes of polystochastic matrices}
\date{June 20, 2024}

\maketitle

\begin{abstract}
A multidimensional nonnegative matrix is called polystochastic if the sum of entries in each  line is equal to $1$. The set of all polystochastic matrices of order $n$ and dimension $d$ is a convex polytope $\Omega_n^d$. 

In the present paper, we compare known bounds on the number of vertices of the polytope $\Omega_n^d$ and prove that the number of vertices of  $\Omega_3^d$ is doubly exponential on $d$.


\end{abstract}

\section{Introduction and definitions}

Polystochastic matrices are a natural extension of doubly stochastic matrices to greater dimensions.  The properties of the convex polytope of doubly stochastic matrices were extensively studied by Brualdi and Gibson in the 1970s~\cite{BruGib.doublypolyhed, BruGib.doublypolyII, BruGib.doublypolyIII}, while there are still few results on  the  more  complicated  polytope of polystochastic matrices.  

 Knowledge  of the vertex set of a convex polytope allows one  to reveal its geometrical structure and simplify the solution of some optimization problems.  The Birkhoff theorem states that the vertices  of the polytope of doubly stochastic matrices  are the permutation matrices only, while the polytope of polystochastic matrices has many other vertices for which we do not have  good descriptions nor exact bounds on their number.   The aim of the present paper is to estimate  the number of vertices of the polytope of polystochastic matrices.

Let us give the necessary definitions. 
A \textit{$d$-dimensional matrix $A$ of order $n$} is an array $(a_\alpha)_{\alpha \in I^d_n}$, $a_\alpha \in\mathbb R$, whose entries are indexed by $\alpha$ from the index set $I_n^d = \{  \alpha = (\alpha_1, \ldots, \alpha_d ) | \alpha_i \in \{ 1, \ldots, n\}\}$.  A matrix $A$ is called \textit{nonnegative} if all $a_{\alpha} \geq 0$, and it is a \textit{$(0,1)$-matrix} if all its entries are $0$ or $1$. The \textit{support} $supp(A)$ of a matrix $A$ is the set of all indices $\alpha$ for which $a_{\alpha} \neq 0$. 

Given $k\in \left\{0,\ldots,d\right\}$, a \textit{$k$-dimensional plane} in $A$ is the submatrix  obtained by fixing $d-k$ positions in indices and letting  the values in other $k$ positions vary from $1$ to $n$. We will say that the set of fixed positions defines the direction of a plane.   A $1$-dimensional plane is said to be a \textit{line}. Matrices $A$ and $B$ are called \textit{equivalent} if one can be obtained from the other by transposes (permutations of components of indices) and permutations of  parallel $(d-1)$-dimensional planes.

A multidimensional nonnegative matrix $A$ is called \textit{polystochastic} if the sum of its entries at each line is equal to $1$.  Polystochastic matrices of dimension $2$ are known as \textit{doubly stochastic}.  Since doubly stochastic $(0,1)$-matrices are exactly the permutation matrices,   for  $d \geq 3$  we  will say that $d$-dimensional polystochastic $(0,1)$-matrices are \textit{$d$-dimensional (or multidimensional) permutations}. There is a one-to-one correspondence between $d$-dimensional permutations of order $n$ and $(d-1)$-dimensional latin hypercubes of order $n$ that are  $(d-1)$-dimensional matrices of order $n$ filled by $n$ symbols so that each line contains exactly one symbol of each type (for more details, see, for example,~\cite{JurRys.stochmatr}).

It is easy to see that the set of $d$-dimensional polystochastic matrices of order $n$ is a convex polytope that we denote as $\Omega_n^d$ and call the \textit{Birkhoff polytope}.   Under a \textit{dimension} of  $\Omega_n^d$ we mean its geometric dimension as a polytope in $\mathbb{R}^{n^d}$, and \textit{facets} are its faces of one less dimension than the polytope itself.

A matrix $A \in \Omega_n^d$ is a \textit{vertex} of the Birkhoff polytope $\Omega_n^d$ if there are no matrices $B_1, B_2 \in \Omega_n^d$ such that $A = \lambda B_1 + (1 - \lambda) B_2$ for some $ 0 < \lambda <1$.  The definition also implies that for every $d$-dimensional  polystochastic matrix $A$ of order $n$ there is a decomposition of the form $A = \sum\limits_{i} \lambda_i B_i$, where $\lambda_i > 0$, $\sum\limits_i \lambda_i = 1$, and $B_i$ are  vertices of $\Omega_n^d$ such that $supp(B_i) \subseteq A$.  Let $V(n,d)$ denote the number of vertices of the polytope $\Omega_n^d$. Note that every multidimensional permutation is a vertex in $\Omega_n^d$. 
 
At last, we will say that a multidimensional matrix  is a \textit{zero-sum} matrix if the sum of entries at each line of $A$ is equal to $0$. For example, the difference between two polystochastic matrices of the same order and dimension is a zero-sum matrix. 

The structure of the paper is as follows. In Section~\ref{boundsec}, using a general bound on the number of faces in polytopes,  we get an upper bound on the number of vertices of   the polytope  of polystochastic matrices. Then we summarize other known bounds on the number of vertices of $\Omega_n^d$  and analyze their asymptotic  behavior  when the order $n$ or the dimension $d$ of matrices is fixed.  In particular, we see that for the number $V(3,d)$ of the polytope of $d$-dimensional matrices of order $3$, the lower and upper bounds differ dramatically. To narrow this gap, in Section~\ref{constrsec} we propose a  construction of vertices of $\Omega_3^d$ that shows that $V(3,d)$ grows doubly exponentially.

\section{Bounds on the number of vertices of $\Omega_n^d$} \label{boundsec}

We start with an upper bound on the vertices in a general polytope.  The well-known result of McMullen~\cite{McMullen.facesofpolytope} states that cyclic polytopes have the largest possible number of faces among all convex polytopes with a given dimension and number of vertices.  As a consequence, one can estimate the number of vertices of a polytope with a given dimension and number of facets.

\begin{utv}[see, e.g., \cite{Bron.covexpoly}] \label{McMull} 
The number of vertices $V$ of a convex $m$-dimensional polytope with $k$ facets, $k \geq m$, is
$$V \leq {k - \lfloor \frac{m+1}{2} \rfloor \choose k - m} +  {k - \lfloor \frac{m+2}{2} \rfloor \choose k - m}. $$
\end{utv}

The polytope of $d$-dimensional polystochastic matrices of order $2$ has dimension $1$  and  only two vertices, i.e., the  multidimensional permutations (see, for example,~\cite{my.obzor}).  But for $n \geq 3$, the polytope $\Omega_n^d$ is nontrivial.

\begin{utv} \label{polytopedim}
Let $n \geq 3$. The polytope $\Omega_n^d$ is a $(n-1)^d$-dimensional polytope  with $n^d$ facets. 
\end{utv}

\begin{proof}
Similar to the polytope $\Omega^2_n$ of doubly stochastic matrices (see, e.g.,~\cite{BruGib.doublypolyhed}), every face $F$ of the polytope $\Omega_n^d$ is defined by a set of indices in which a matrix $A$ from $F$ takes zero values. So the facets of $\Omega_n^d$ are $\{ A \in  \Omega_n^d| a_{\alpha} = 0 \}$ for some  $\alpha \in I_n^d $, and  there are $n^d$ facets in $\Omega_n^d$. 

To find  the dimension of $\Omega_n^d$, it is sufficient to note that the space of $d$-dimensional zero-sum matrices of order $n$ has dimension  $(n-1)^d$ because every such matrix is  uniquely defined by  values in any $d$-dimensional submatrix of order $n-1$. 
\end{proof}

From Propositions~\ref{McMull} and~\ref{polytopedim}, we get the following upper bound on the number of vertices of $\Omega_n^d$. For $d=3$ it was previously stated in \cite{Li2Zhang.vertstoch}. 

\begin{teorema} \label{polytopeupperbound}
For the number of vertices $V(n,d)$ of the polytope of polystochastic matrices $\Omega_n^d$ we have
$$V(n,d) \leq {n^d - \lfloor \frac{(n-1)^d+1}{2} \rfloor \choose n^d - (n-1)^d} +  {n^d - \lfloor \frac{(n-1)^d+2}{2} \rfloor \choose n^d - (n-1)^d}. $$
\end{teorema}

To our knowledge, there are no upper bounds on the number of vertices of $\Omega_n^d$ that use the specific properties of this polytope. So finding any improvement to Theorem~\ref{polytopeupperbound} is an interesting question. 
 
A natural lower bound on  the number $V(n,d)$  of vertices  of $\Omega_n^d$ is  the number of multidimensional permutations because every $d$-dimensional permutation of order $n$ is a vertex of $\Omega_n^d$.

Let us study the asymptotics of the number of vertices $V(n,d)$ when $d$ is fixed and $n \rightarrow \infty$. 

When $d = 2$, the well-known Birkhoff theorem states that every vertex of the polytope of doubly stochastic matrices is a permutation matrix.  So $V(n,2) = n!$ that is the number of permutation matrices of order $n$.

In~\cite{Keevash.existdesII} Keevash found the lower bound on  the number of multidimensional permutations of fixed dimension, which with the upper bound by Linial and Luria~\cite{LinLur.hdimper} gives the following.

\begin{teorema}[\cite{Keevash.existdesII}, \cite{LinLur.hdimper}] \label{permutasym}
The number of $d$-dimensional permutations of order $n$ is 
$\left( \frac{n}{e^{d-1}}  + o(n)\right)^{n^{d-1}} $
as $d \geq 2$ is fixed and $n \rightarrow \infty$. 
\end{teorema}

For $d \geq 3$ and $n \geq 3$, the polytope $\Omega_n^d$ has vertices other than multidimensional permutations,  but we know not many examples  and very few  constructions of such vertices. Most of these constructions~\cite{CuiLiNg.Birkfortensor, LinLur.birvert, FichSwart.3dimstoch} produce vertices of $\Omega_n^d$ that have exactly two $\nicefrac{1}{2}$-entries in each line. The only improvement on the lower bound of vertices of the polytope of $\Omega_n^d$ of fixed dimension was obtained for $d = 3$ by Linial and Luria in~\cite{LinLur.birvert}.

\begin{teorema}[\cite{LinLur.birvert}] \label{d3lowangles}
If $M(n,3)$ is the number of $3$-dimensional permutations of order $n$, then for the number $V(n,3)$ of vertices of $\Omega_n^d$ we have
$$V(n,3) \geq M(n,3)^{3/2 - o(1) } \mbox{ as } n \rightarrow \infty.$$
\end{teorema}

Summarizing  these results, we deduce the following asymptotic bounds for the logarithm of $V(n,d)$ when $d$ is fixed.
\begin{utv}
If $d \geq 4$ is fixed, then for the number $V(n,d)$ of vertices of $\Omega_n^d$ we have
$$ n^{d-1} \ln n \cdot (1 + o(1)) \leq \ln  V(n,d) \leq  dn^{d-1} \ln n \cdot (1 + o(1)) $$
as $n \rightarrow \infty$.
In addition, if $d = 3$, then 
$$ \frac{3}{2} n^{2} \ln n  \cdot  (1 + o(1)) \leq \ln V(n,3) \leq  3 n^{2} \ln n \cdot (1 + o(1)), $$
and if $d = 2$, then $\ln V(n,2) = \ln  n! = n \ln n  \cdot(1 + o(1))$. 
\end{utv}

\begin{proof}
For $d \geq 3$, all upper bounds follow from the standard estimation ${m \choose k} \leq \frac{m^k}{k!}$ of the binomial coefficients in Theorem~\ref{polytopeupperbound} and further analyzes of the expressions for large $n$.

For $d \geq 4$, the lower bound follows from the estimation of the number of multidimensional permutations (Theorem~\ref{permutasym}), and for $d = 3$ it is improved by Theorem~\ref{d3lowangles}.

At last, the case $d = 2$ is the Birkhoff theorem for doubly stochastic matrices.
\end{proof}

A comparison of  weaker lower and upper bounds on the number of vertices of $3$-dimensional polystochastic matrices was given in~\cite{Zhangx2.extrempoints}.

On the basis of  these bounds, we propose the following conjecture.

\begin{con}
For every $d \geq 2$, there is a constant $c_d$, $1 \leq c_d \leq d$, such that for the number $V(n,d)$ of vertices of $\Omega_n^d$ we have
$$\ln V(n,d) = c_d n^{d-1} \ln n \cdot (1 + o(1)). $$
\end{con}

Let us turn to the case when the order  of polystochastic matrices is fixed but the dimension grows.  

As we noted before, there are only two vertices in the polytope $\Omega_2^d$.  It is also well known that for every $d$, the $d$-dimensional permutation of order $3$ is unique up to the equivalence, and there are $3 \cdot 2^{d-1}$ different such multidimensional permutations. 

The asymptotics of the number of $d$-dimensional permutations of order $4$ were found in~\cite{PotKrot.asymptquasi4}. It gives that  $\log_2 V(4,d) \geq   2^{d-1}  (1 + o(1)) $. 

 Up to date, the best lower bounds on the number of  $d$-dimensional permutations of order $n$, when $n \geq 5$ is fixed, were proved by Potapov and Krotov in~\cite{PotKrot.numbernary}. Their results imply the following.
 
\begin{teorema}[\cite{PotKrot.numbernary}] \label{nfixedpermbound}
Let $V(n,d)$ be the number of vertices of the polytope of $d$-dimensional matrices of order $n$. Then  $\log_2 V(5,d) \geq   3^{(d-1)/3}   (1 - o(1)) $ as $d \rightarrow \infty$,  $\log_2 V(n,d) \geq   {\left( \frac{n}{2} \right) }^{d-1}$ if $n \geq 6$ is even, and  $\log_2 V(n,d) \geq   {\left( \frac{n-3}{2} \right)}^{\frac{d-1}{2}} {\left( \frac{n-1}{2} \right)}^{\frac{d-1}{2}} $ if $n \geq 7$ is odd.
 \end{teorema}
 
Till the present work, there were no rich constructions and lower bounds on the vertices of $\Omega_n^d$ for fixed $n$ that are different from the multidimensional permutations. 

Concerning the upper bound, an expansion of the binomial coefficients in  Theorem~\ref{polytopeupperbound} for fixed $n$ gives the following.

\begin{utv}\label{asymupper}
If $n$ is fixed, then  for the logarithm  of the number $V(n,d)$ of vertices of $d$-dimensional polystochastic matrices of order $n$ we have
$$\log_2 V(n,d) \leq \log_2  \frac{n}{n-1} \cdot  \frac{d}{2}  \cdot  (n-1)^d  (1 + o(1)) \mbox{ as } d \rightarrow \infty.$$
\end{utv}

Thus, we have a substantial difference in lower and upper bounds  for the numbers of vertices of $\Omega_n^d$ when $n$ is fixed, and finding the asymptotic of these numbers is an interesting open problem.

\section{Lower bound on the number of vertices of $\Omega_3^d$} \label{constrsec}

In this section, we prove the following theorem.

\begin{teorema} \label{n3lowerbound}
For the number of vertices $V(3,d)$ we have
$$\log_2 V(3,d) \geq  c 2^{\delta d} (1 + o(1) ),$$
where $c =  \frac{1}{4}\log_2 9/5 \approx 0.212$ and $\delta \approx 0.047$.
\end{teorema}

The main idea of the proof is to construct  a large set $\mathcal{A}$  of $d$-dimensional polystochastic matrices of order $3$ such that for every three matrices $A_1, A_2, A_3 \in \mathcal{A}$ the faces of $\Omega_3^d$ defined by their supports do not share a common vertex.   It means that every vertex from a decomposition of some matrix $A \in \mathcal{A}$ into a convex sum of vertices can appear  in at most one  other decomposition of some matrix $B \in \mathcal{A}$, and, therefore, we have at least $|\mathcal{A}| / 2$ vertices   of $\Omega_3^d$.

But firsty we aim to construct a rich set of polystochastic matrices. 
For this purpose  we need  several auxiliary notions and definitions.

Given indices $\alpha, \beta \in I_n^d$,  let $\rho(\alpha, \beta)$ denote the Hamming distance between $\alpha$ and $\beta$, i.e.,  the number of  positions in which these indices differ. 

For an index $\alpha \in I_3^d$, define the set of indices $T_{\alpha}  = \{ \beta| \rho(\alpha, \beta) = d \}$.   In other words, $T_\alpha$ is the $d$-dimensional submatrix of order $2$  formed by indices at the maximal distance from the index $\alpha$. 

Let $S \subseteq I_3^d$ be a subset of indices of a $d$-dimensional matrix of order $3$. Given $\varepsilon$, $0 < \varepsilon < 1$,  will say that  the set $S$  is \textit{$\varepsilon$-sparse} if 
\begin{enumerate}
\item for all $\alpha, \beta \in S$, we have $\rho(\alpha, \beta) \geq \varepsilon d$;
\item for every $\alpha \in S$ there is an index $\gamma_\alpha \in I_3^d$ such that $T_{\gamma_\alpha} \cap S = \{ \alpha\}$. 
\end{enumerate}

We are going to use $\varepsilon$-sparse sets as the complements of the supports of polystochastic matrices of order $3$.  But firstly, we show that $\varepsilon$-sparse sets exist and can be quite large.  For this purpose, we need one well-known statement of  the coding theory that follows from the works of Shannon~\cite{shannon.optcode}.

\begin{utv}[see~\cite{shannon.optcode}]   \label{binomentrop}
Let $0 \leq \varepsilon \leq 1/2$ and $H(\varepsilon) = - \varepsilon \log_2 \varepsilon - (1 - \varepsilon) \log_2 (1 - \varepsilon)$ be the binary entropy. Then for every $d$ and $N \leq 2^{d (1 - H(\varepsilon))}$ there is a set $S \subseteq I_2^d$ of size $N$  such that for every $\alpha, \beta \in S$ we have $\rho(\alpha, \beta) \geq \varepsilon d$.
\end{utv}

\begin{utv} \label{sparseexist}
Let $0 < \varepsilon  \leq 1/2$.  Then for every  $N \leq 2^{(1 - H(\varepsilon)) d}$ there exists an $\varepsilon$-sparse $S$ set of size $N$ in $I_3^d$.
\end{utv}

%Let  $S$ be a  subset  of $T_{\textbf{0}}$, where  $T_{\textbf{0}} = \{ \alpha \in I_3^d : \alpha_i \in \{ 1, 2\}\}$,  such that  for all $\alpha, \beta \in S$ it holds $\rho(\alpha, \beta) \geq  \varepsilon d$. We show that a maximal (by inclusion) set $S$ with such property has size at least $2^{(1 - H(\varepsilon)) d} -1$. 

%Note that the number of indices $\gamma \in T_{\textbf{0}}$ such that $\rho (\gamma,S) < k$ is not greater than $|S| \sum\limits_{k < \varepsilon d} {k \choose d} $, that, by Proposition~\ref{binomentrop}, it is not greater than $ |S| \cdot  2^{H(\varepsilon) d}$. Consequently,  if $|S| \leq 2^{-H(\varepsilon) d } (2^d -1) $, then there exists an index $\gamma \notin S$, $\gamma \in T_{\textbf{0}}$, such that $\rho (S, \gamma) \geq \varepsilon d$, and so $S$ is not maximal.

\begin{proof}
For shortness, we denote the index $(1 ,\ldots, 1)$ from $I_3^d$ by $\textbf{1}$ and let $T_{\textbf{1}} = \{ \alpha \in I_3^d : \alpha_i \in \{ 2, 3\}\}$. 
By Proposition~\ref{binomentrop}, for every  $N \leq 2^{(1 - H(\varepsilon)) d}$ there is a subset  $S$   of $T_{\textbf{1}}$ such that  for all $\alpha, \beta \in S$ it holds $\rho(\alpha, \beta) \geq  \varepsilon d$. 

Let us show that all such sets $S$ satisfy the second condition of the definition of $\varepsilon$-sparse sets.  Given $\alpha \in T_{\textbf{1}} $, consider an index $\gamma$ from $T_{\textbf{1}}$ such that $\gamma_i = 2$ if $\alpha_i = 3$, and $\gamma_i= 3$ if $\alpha_i = 2$ for every $i \in \{ 1,\ldots, d\}$.  It is easy to see that   $T_{\gamma} \cap T_{\textbf{1}} = \{ \alpha\}$. Since $S \subseteq T_{\textbf{1}}$, for every $\alpha \in S$ we have that   $T_{\gamma} \cap S = \{ \alpha\}$, so we can take $\gamma$ as $\gamma_\alpha$.
\end{proof}

Now we prove that  for every $\varepsilon$-sparse set $S$ (if it has not  very large size and $d$ is not very small), we can find a polystochastic matrix of order $3$ whose  complement of the support is exactly the set $S$.

\begin{lemma} \label{sparsetopoly}
Let $d \geq 14 / \varepsilon$ and $S \subseteq I_3^d$ be an $\varepsilon$-sparse set  of size $N = |S|  \leq 2^{\varepsilon d /4}$. Then there exists a  $d$-dimensional polystochastic matrix  $A$ of order $3$ such that $supp(A) = I_3^d \setminus S$. 
\end{lemma}

\begin{proof}
We will look for a $d$-dimensional zero-sum matrix $M$ of order $3$ such that $m_\alpha = 1$ for all $\alpha \in S$ and $|m_{\alpha}| \leq 3/4$ for all $\alpha \not\in S$. Then the matrix $A = J  - \frac{1}{3} M$ is the required polystochastic matrix, where $J$ is the $d$-dimensional polystochastic matrix of order $3$, whose all entries are equal to $1/3$.

We construct the matrix $M$ in two steps.

1. Given index $\alpha \in I_3^d$, define the  matrix $F^{\alpha} = (f^\alpha_\beta)_{\beta \in I_n^d}$  with entries $f^\alpha_\beta = \left( - \frac{1}{2} \right)^{\rho(\alpha,\beta)}$. Note that all matrices $F^\alpha$ are zero-sum because each line consists of indices $\beta_1, \beta_2, \beta_3$ such that $\rho (\alpha, \beta_1) = k$ and  $\rho (\alpha, \beta_2) = \rho(\alpha, \beta_3) = k +1$ for some $k \in \{ 0, \ldots d-1\}$. 

Consider the zero-sum matrix
$$M' = \sum\limits_{\alpha \in S} F^\alpha$$
and  estimate its entries.  For each  $\alpha \in S$, denote $\delta_{\alpha} = 1 - m'_\alpha$.  Since $S$ is an $\varepsilon$-sparse set of size $N$,  we have that 
$$|\delta_\alpha| = |1 - m'_\alpha| \leq \frac{N}{2^{\varepsilon d}} \leq 2^{-3 \varepsilon d /4},$$
because for all $\gamma \in S$, $\gamma \neq \alpha$, each matrix $F^\gamma$ has the absolute value $f^{\gamma}_\alpha$ in index $\alpha$ not greater than $1 / 2^{\varepsilon d}$.   In particular, for $\delta = \max\limits_{\alpha \in S}  |\delta_\alpha|$ we have $\delta  \leq 2^{ - 3 \varepsilon d /4} $. 

Suppose now that $\alpha \not\in S$ and  $\gamma \in S$ is an index such that $\rho(\alpha, \gamma)$ is minimal.  Then from the definition of an $\varepsilon$-sparse set, for every other $\beta \in S$, $\beta \neq \gamma$, we have that $\rho (\alpha, \beta) \geq (\varepsilon d - 1)/ 2$. Using $N \leq 2^{\varepsilon d /4}$, we obtain
 $$|m'_\alpha |   \leq  |f^\gamma_\alpha| + \sum\limits_{\beta \in S, \beta \neq \gamma}  |f^\beta_\alpha| \leq   \frac{1}{2}+  \frac{N}{2^{(\varepsilon d-1) /2}} \leq  \frac{1}{2} + 2^{- \varepsilon d / 4 + 1/2} \leq \frac{5}{8}, $$
 because $2^{- \varepsilon d / 4 + 1/2}  \leq 1/8$, when $d \geq 14/ \varepsilon$.
 
2.  Now we modify the  matrix $M'$ to obtain the required matrix $M$. 

Since the set $S$ is $\varepsilon$-sparse, for every $\alpha \in S$ there is $\gamma_\alpha \in I_3^d$ such that $T_{\gamma_\alpha} \cap S = \{ \alpha\}$.  Let $R^\alpha$ be the zero-sum matrix such that for every $\beta \in T_{\gamma_\alpha}$ the entry $r^\alpha_\beta = (-1)^{\rho(\alpha, \beta)}  \cdot \delta_\alpha $ and for every $\beta \not\in T_{\gamma_\alpha} $ we put $r^\alpha_\beta = 0$. Consider the matrix
$$M = M' + \sum\limits_{\alpha \in S} R^\alpha.$$
 
Using the definition of $\delta_\alpha$ and the fact that  for every $\alpha \in S$ it holds $T_{\gamma_\alpha} \cap S = \{ \alpha\}$, we see that $m_\alpha = 1$ for all $\alpha \in S$ as required.
 
Suppose that  $\alpha \not\in S$. Using inequalities  $N \leq 2^{\varepsilon d/4}$, $\delta \leq 2^{- 3 \varepsilon d/ 4}$, and $d \geq 14 / \varepsilon$, we obtain
 $$|m_\alpha| \leq |m'_\alpha| + \sum\limits_{\alpha \in S} |\delta_\alpha| \leq \frac{5}{8} + N\delta \leq \frac{5}{8} +  2^{-\varepsilon d / 2} \leq \frac{3}{4}.$$
\end{proof}

To construct many polystochastic matrices with the desired property,  we  also utilize $3$-perfect hash codes  that are known as trifferent codes.

Given $q \geq 2$ and $N \in \mathbb{N}$, a \textit{$q$-perfect hash code $C$ of block length $N$} is a collection $C$ of words of length $N$ under alphabet $\{ 1, \ldots, q\}$ such that for any distinct $q$ words $w_1, \ldots, w_q$  from $C$ there is a position $i$, $i \in \{1, \ldots, N \}$, for which $\{ w_j(i) \, | \,1 \leq j  \leq q \} = \{ 1, \ldots, q\}$.  In what follows, we identify words from $q$-perfect hash codes with indices from $I_q^N$. 

A set $C \subseteq I_3^N$ is called a \textit{trifferent code} if it is a $3$-perfect hash code.  Up to date, the asymptotically large trifferent codes were constructed by  K\"{o}rner  and Marton in~\cite{KorMar.perfhash}.

\begin{teorema}[\cite{KorMar.perfhash}, Theorem 1] \label{dispersedset}
For every $N \in \mathbb{N}$, there exists a trifferent code $C$ in $I_3^N$ of size $|C| =  \left(  \frac{9}{5}\right)^{N/4}$.
\end{teorema}

Now we are ready to prove the main result of this section.

\begin{proof}[Proof of Theorem~\ref{n3lowerbound}]

Recall that we aim to construct  a large set $\mathcal{A}$  of $d$-dimensional polystochastic matrices of order $3$ such that for every three matrices $A_1, A_2, A_3 \in \mathcal{A}$ there are no vertices $B$ of $\Omega_3^d$ for which $supp(B) \subseteq supp (A_1) \cap supp(A_2) \cap supp(A_3)$. 

Let $\varepsilon \leq 1/2 $ be a solution of  $1 - H(\varepsilon) = \varepsilon /8 $, $\varepsilon \approx 0.3735$,  and put $\mu =  \varepsilon /8$.  Then by Proposition~\ref{sparseexist}, there exists an $\varepsilon$-sparse set $S$ in $I_3^d$ with cardinality $N = \lfloor 2^{\mu d} \rfloor$, and by Theorem~\ref{dispersedset}, there is a trifferent code $C$ in $I_3^N$ such that $|C| = \left(  \frac{9}{5}\right)^{N/4}$. 

Using the $\varepsilon$-sparse set $S$ in $I_3^d$  and the trifferent code $C$, let us now construct many sparse sets in $I_3^{d+1}$. Since a word $x \in C$ has length $N$, we may assume that its positions are indexed by $\alpha \in S$. For every $x \in C$, consider the set $S_x$ in $I_3^{d+1}$ such that $S_{x} = \{ (\alpha, x_{\alpha}) : \alpha \in S \}$, $|S_x| = |S|$. Also note that the number of sets $S_x$ is equal to $|C|$.

Let us prove that all sets $S_{x}$ are $\varepsilon/2$-sparse. 
Let  $\beta = (\alpha, x_\alpha)$  and $\beta' =  (\alpha', x_{\alpha'})$, where $\beta, \beta' \in S_x$. Then $\rho(\beta, \beta') \geq \rho(\alpha, \alpha')$. Since  $S$ is an $\varepsilon$-sparse set, we have that $\rho(\alpha, \alpha') \geq \varepsilon d \geq\frac{\varepsilon}{2} (d+1) $, and so $\rho(\beta, \beta') \geq \frac{\varepsilon}{2} (d+1)$.

Next, by the condition on the $\varepsilon$-sparse set $S$, for every $\alpha \in S$ there is $\gamma_{\alpha} \in I_3^d$ such that $T_{\gamma_{\alpha}} \cap S = \{\alpha\}$.  Given  $\beta \in S_x$, $\beta = (\alpha, x_\alpha)$, consider index $\gamma_{\beta} = (\gamma_{\alpha}, j)$ from $I_3^{d+1}$, where $j \in \{ 1,2, 3\}$ is different from $ x_\alpha$. Then the construction of sets $S_x$ implies that $T_{\gamma_\beta} \cap S_{x} = \{ \beta\}$.  Therefore, sets $S_{x}$ are $\varepsilon/2$-sparse for all $x \in C$. 

Using Lemma~\ref{sparsetopoly} and the fact that $N \leq 2^{\varepsilon d/8}$, for all $d \geq 28/\varepsilon$ and  each $x \in C$ there is a $(d+1)$-dimensional polystochastic matrix $A_{x}$ of order $3$ such that $supp(A_x) = I_3^{d+1} \setminus S_x$.  Define a collection of polystochastic matrices $\mathcal{A} = \{ A_x : x \in C\}$.

Since $C$ is a trifferent code,   for all words $x^1, x^2, x^3 \in C$ there is  a position in which all these three words are different. In our construction, this position corresponds to a  line   $\ell$ of direction $d +1$   in $I_3^{d+1}$ such that $supp(A_{x^1}) \cap supp(A_{x^2})  \cap supp(A_{x^3})  \cap  \ell = \emptyset $.  Therefore, there are   no vertices $B$ of the Birkhoff polytope $\Omega_3^{d+1}$  for which  $supp(B) \subseteq supp(A_{x^1}) \cap supp(A_{x^2})  \cap supp(A_{x^3})$.  

Denote by $\mathcal{B}$ the set of all vertices $B$ of the polytope $\Omega_3^{d+1}$ such that $supp (B) \subseteq A_{x}$ for some $A_x \in \mathcal{A}$. The obtained property of the set $\mathcal{A}$ means that  for every $B \in \mathcal{B}$ there are at most two matrices $A_{x}$ such that $supp (B) \subseteq A_{x}$. Therefore, 
$$|\mathcal{B}| \geq \frac{|C|}{2} = \frac{1}{2} \left(  \frac{9}{5}\right)^{N/4} \geq  \frac{1}{2} \left(  \frac{9}{5}\right)^{ \frac{\lfloor 2^{ \mu d}  \rfloor}{4}}  $$
that implies the statement of the theorem.
\end{proof}

\textbf{Remark.} The same reasoning can be applied for the construction of quite rich families of vertices of polytopes $\Omega_n^d$ for any  fixed $n \geq 3$ and large $d$. But for $n \neq 3$ the numbers of such vertices will be much less than the lower bounds on the numbers of multidimensional permutations from Theorem~\ref{nfixedpermbound}.

\section*{Acknowledgments}

We are grateful to Dmitriy Zakharov for the reference to the trifference problem and trifferent codes.

The work of Anna Taranenko  was supported by the Russian Science Foundation under grant No~22-21-00202, https://rscf.ru/project/22-21-00202/.


\bib{Bron.covexpoly}{book}{
  author={Br\o ndsted, Arne},
  title={An introduction to convex polytopes},
  series={Graduate Texts in Mathematics},
  volume={90},
  publisher={Springer-Verlag, New York-Berlin},
  date={1983},
  pages={viii+160},
  isbn={0-387-90722-X},
}

\bib{BruGib.doublypolyhed}{article}{
  author={Brualdi, Richard A.},
  author={Gibson, Peter M.},
  title={Convex polyhedra of doubly stochastic matrices. I. Applications of the permanent function},
  journal={J. Combinatorial Theory Ser. A},
  volume={22},
  date={1977},
  number={2},
  pages={194--230},
  issn={0097-3165},
  review={\MR {437562}},
  doi={10.1016/0097-3165(77)90051-6},
}

\bib{BruGib.doublypolyII}{article}{
  author={Brualdi, Richard A.},
  author={Gibson, Peter M.},
  title={Convex polyhedra of double stochastic matrices. II. Graph of $U\sb {n}$},
  journal={J. Combinatorial Theory Ser. B},
  volume={22},
  date={1977},
  number={2},
  pages={175--198},
  issn={0095-8956},
  doi={10.1016/0095-8956(77)90010-7},
}

\bib{BruGib.doublypolyIII}{article}{
  author={Brualdi, Richard A.},
  author={Gibson, Peter M.},
  title={Convex polyhedra of doubly stochastic matrices. III. Affine and combinatorial properties of $U\sb {n}$},
  journal={J. Combinatorial Theory Ser. A},
  volume={22},
  date={1977},
  number={3},
  pages={338--351},
  issn={0097-3165},
  doi={10.1016/0097-3165(77)90008-5},
}

\bib{CuiLiNg.Birkfortensor}{article}{
  author={Cui, Lu-Bin},
  author={Li, Wen},
  author={Ng, Michael K.},
  title={Birkhoff--von Neumann theorem for multistochastic tensors},
  journal={SIAM J. Matrix Anal. Appl.},
  volume={35},
  date={2014},
  number={3},
  pages={956--973},
  issn={0895-4798},
  doi={10.1137/120896499},
}

\bib{FichSwart.3dimstoch}{article}{
  author={Fischer, P.},
  author={Swart, E. R.},
  title={Three-dimensional line stochastic matrices and extreme points},
  journal={Linear Algebra Appl.},
  volume={69},
  date={1985},
  pages={179--203},
  issn={0024-3795},
  doi={10.1016/0024-3795(85)90075-8},
}

\bib{JurRys.stochmatr}{article}{
  author={Jurkat, W. B.},
  author={Ryser, H. J.},
  title={Extremal configurations and decomposition theorems. I},
  journal={J. Algebra},
  volume={8},
  date={1968},
  pages={194--222},
  issn={0021-8693},
  doi={10.1016/0021-8693(68)90045-8},
}

\bib{Keevash.existdesII}{article}{
  author={Keevash, Peter},
  title={The existence of designs II},
  eprint={https://arxiv.org/abs/1802.05900},
}

\bib{KorMar.perfhash}{article}{
  author={K\"{o}rner, J.},
  author={Marton, K.},
  title={New bounds for perfect hashing via information theory},
  journal={European J. Combin.},
  volume={9},
  date={1988},
  number={6},
  pages={523--530},
  issn={0195-6698},
  doi={10.1016/S0195-6698(88)80048-9},
}

\bib{Li2Zhang.vertstoch}{article}{
  author={Li, Zhongshan},
  author={Zhang, Fuzhen},
  author={Zhang, Xiao-Dong},
  title={On the number of vertices of the stochastic tensor polytope},
  journal={Linear Multilinear Algebra},
  volume={65},
  date={2017},
  number={10},
  pages={2064--2075},
  issn={0308-1087},
  doi={10.1080/03081087.2017.1310178},
}

\bib{LinLur.hdimper}{article}{
  author={Linial, Nathan},
  author={Luria, Zur},
  title={An upper bound on the number of high-dimensional permutations},
  journal={Combinatorica},
  volume={34},
  date={2014},
  number={4},
  pages={471--486},
  issn={0209-9683},
  doi={10.1007/s00493-011-2842-8},
}

\bib{LinLur.birvert}{article}{
  author={Linial, Nathan},
  author={Luria, Zur},
  title={On the vertices of the $d$-dimensional Birkhoff polytope},
  journal={Discrete Comput. Geom.},
  volume={51},
  date={2014},
  number={1},
  pages={161--170},
  issn={0179-5376},
  doi={10.1007/s00454-013-9554-5},
}

\bib{McMullen.facesofpolytope}{article}{
  author={McMullen, P.},
  title={The maximum numbers of faces of a convex polytope},
  journal={Mathematika},
  volume={17},
  date={1970},
  pages={179--184},
  issn={0025-5793},
  doi={10.1112/S0025579300002850},
}

\bib{PotKrot.asymptquasi4}{article}{
  author={Potapov, V. N.},
  author={Krotov, D. S.},
  title={Asymptotics of the number of $n$-quasigroups of order $4$},
  language={Russian, with Russian summary},
  journal={Sibirsk. Mat. Zh.},
  volume={47},
  date={2006},
  number={4},
  pages={873--887},
  issn={0037-4474},
  translation={ journal={Siberian Math. J.}, volume={47}, date={2006}, number={4}, pages={720--731}, issn={0037-4466}, },
  doi={10.1007/s11202-006-0083-9},
}

\bib{PotKrot.numbernary}{article}{
  author={Potapov, V. N.},
  author={Krotov, D. S.},
  title={On the number of $n$-ary quasigroups of finite order},
  journal={Diskret. Mat.},
  volume={24},
  date={2012},
  number={1},
  pages={60--69},
  issn={0234-0860},
  translation={ journal={Discrete Math. Appl.}, volume={21}, date={2011}, number={5-6}, pages={575--585}, },
  doi={10.1515/dma.2011.035},
}

\bib{shannon.optcode}{article}{
  author={Shannon, C. E.},
  title={A mathematical theory of communication},
  journal={Bell System Tech. J.},
  volume={27},
  date={1948},
  pages={379--423, 623--656},
  issn={0005-8580},
  review={\MR {26286}},
  doi={10.1002/j.1538-7305.1948.tb01338.x},
}

\bib{my.obzor}{article}{
  author={Taranenko, Anna A.},
  title={Permanents of multidimensional matrices: properties and applications},
  language={Russian, with English and Russian summaries},
  journal={Diskretn. Anal. Issled. Oper.},
  volume={23},
  date={2016},
  number={4},
  pages={35--101},
  issn={1560-7542},
  translation={ journal={J. Appl. Ind. Math.}, volume={10}, date={2016}, number={4}, pages={567--604}, issn={1990-4789}, },
  doi={10.1134/S1990478916040141},
}

\bib{Zhangx2.extrempoints}{article}{
  author={Zhang, Fuzhen},
  author={Zhang, Xiao-Dong},
  title={Enumerating extreme points of the polytopes of stochastic tensors: an optimization approach},
  journal={Optimization},
  volume={69},
  date={2020},
  number={4},
  pages={729--741},
  issn={0233-1934},
  doi={10.1080/02331934.2019.1647198},
}



\begin{bibdiv}
    \begin{biblist}[\normalsize]
    \bibselect{biblio}
    \end{biblist}
    \end{bibdiv}
\end{document}